\documentclass[publish,b5paper,twoside]{amsart}
\usepackage[dvipdfmx]{graphicx,color}
\usepackage{fullpage}
\usepackage{amssymb}
\usepackage{amsmath}
\usepackage{amscd}
\usepackage{amstext}
\usepackage{amsfonts}
\usepackage[all]{xy}
\usepackage{comment}

\usepackage{amsmath,amssymb}
\usepackage{mathtools}
\usepackage{amsthm}
\usepackage{comment}
\usepackage{enumerate}

\newcommand{\C}{\mathbb{C}}
\newcommand{\R}{\mathbb{R}}

\newcommand{\Z}{\mathbb{Z}}
\newcommand{\g}{\frak{g}}

\def\R{{\mathbb R}}
\def\C{{\mathbb C}}
\def\Z{{\mathbb Z}}
\def\g{\mathfrak{g}}

\newtheorem{dfn}{Definition}[section] 
\newtheorem{thm}[dfn]{Theorem} 
\newtheorem{prop}[dfn]{Proposition} 
\newtheorem{cor}[dfn]{Corollary} 
\newtheorem{lemma}[dfn]{Lemma} 
\newtheorem{example}[dfn]{\it Example} 

\newtheorem{remark}[dfn]{\it Remark} 
\title[Uniformizations of compact Sasakian manifolds]{Uniformizations of compact Sasakian manifolds}

\author{Hisashi Kasuya}
\address{
Department of Mathematics, Graduate School of Science, Osaka University, Osaka,
Japan
}
\email{kasuya@math.sci.osaka-u.ac.jp}

\author{Natsuo Miyatake}
\address{Mathematical Science Center for Co-creative Society, Tohoku University, 468-1 Aramaki Azaaoba, Aoba-ku, Sendai 980-0845, Japan.}
\email{natsuo.miyatake.e8@tohoku.ac.jp}
\email{natsuo.m.math@gmail.com}

\thanks{This work was partially supported by JSPS KAKENHI Grant Number 19H01787.
}

\subjclass[2010]{14D07, 32Q30, 32L05, 53C25,53C07}
\keywords{Sasakian manifold, Uniformization, Variation of Hodge structures, Higgs bundle}

\begin{document}
\begin{abstract}
We give a criterion for compact Sasakian manifolds to be deformed to Sasakian manifolds which are locally isomorphic to circle bundles of anti-canonical bundles over Hermitian symmetric spaces as 
a Sasakian analogue of Simpson's uniformization results related to variations of Hodge structure and Higgs bundles.
\end{abstract}

\maketitle

\section{Introduction}
We are interested in classifications of compact Sasakian manifolds.
Sasakian geometry is an odd-dimensional counterpart of K\"ahler geometry.
However, in contrast to classifications of compact K\"ahler manifolds up to biholomorphism, the existence of a finite dimensional local moduli space of a compact Sasakian manifold is not given by deformation theory.
Sasakian structures are usually defined by contact CR-structures $(T^{1,0}_M,\eta)$ on odd-dimensional manifolds $M$ satisfying the “normality” condition.
It is known that a space of locally non-isomorphic CR structures on a compact strongly pseudo-convex CR manifold can have arbitrary many parameters (see \cite{BSW}).
For classifying $3$-dimensional compact Sasakian manifolds, it is a great idea to classify Sasakian structures up to deformations of certain types.
In \cite{Bel1, Bel2}, Belgun proves that any $3$-dimensional compact Sasakian manifold is a deformation of a standard one;
locally isomorphic to one of canonical homogeneous (left-invariant) Sasakian structures on $S^{3}$, $\widetilde{SL}_{2}(\R)$, $Nil_{3}$ (see also \cite{Ge}).

Belgun's result is an analogue of uniformizations of compact Riemannian surfaces.
It is natural to expect that we can prove a Sasakian version of uniformization results on higher dimensional compact K\"ahler manifolds in terms of Belgun's arguments.

The most important part of uniformizations of compact Riemannian surfaces is to prove that every Riemannian surface of genus $\ge 2$ admits a metric of constant negative curvature.
In \cite{Hi}, Hitchin reproves this statement as a corollary of the Kobayashi-Hitchin correspondence of Higgs bundles.
In \cite{SimC}, Simpson gives a uniformization theorem of higher dimensional compact K\"ahler manifolds by generalizing Hitchin's idea.
Precisely, Simpson shows that the existence of a uniformizing variation of Hodge structure is a criterion for uniformizing a compact K\"ahler manifold by a Hermitian symmetric space.
By using this, he shows that the stability of a certain Higgs bundle is a criterion for uniformizing a compact K\"ahler manifold by the unit ball.

The main result of this paper is to give a Sasakian version of Simpson's uniformization theorem.
For a Hermitian symmetric space $D$ of non-compact type, we can define the homogeneous Sasakian structure $(T_{D}^{1,0},\eta_{D})$ on the circle bundle $S^{1}(\bigwedge^m T^{1,0}_{D})$ of the determinant line bundle of the holomorphic tangent bundle of $D$.
We give a criterion for uniformizing compact Sasakian manifolds by the homogeneous Sasakian structure $(T_{D}^{1,0},\eta_{D})$.

\begin{thm}
Let $(M, T^{1,0}_{M}, \eta)$ be a compact Sasakian manifold.
If $(M,T^{1,0}_{M}, \eta)$ admits a uniformizing $G$-variations of Hodge structure $(G,\rho, h)$ and $c_{1,B}(T_{M})=-C[d\eta]$ for a positive constant $C$, then
the universal covering $\widetilde{M}$ with the lifting of some $A^1_{B}$-deformation of $( T^{1,0}_{M}, \eta)$ is almost isomorphic to $\widetilde{S^{1}}(\bigwedge^m T^{1,0}_{D})$ with the lifting of the homogeneous Sasakian structure $(T^{1,0}_{D},\eta_{D})$
where $\widetilde{S^{1}}(\bigwedge^m T^{1,0}_{D})$ is the universal covering of $S^{1}(\bigwedge^m T^{1,0}_{D})$.

\end{thm}
$c_{1,B}(T_{M})$ is the first basic Chern class of $(M, T^{1,0}_{M}, \eta)$ which is also very important for studying Sasaki-Einstein geometry (see \cite{BG}).
Note that the condition $c_{1,B}(T_{M})=C[d\eta]$ for a positive constant $C$ is a criterion for the existence of a Sasaki-Einstein metric on a compact Sasakian manifold in contrast to our result.
An {\em $A^1_{B}$-deformation} is a slight modification of a second type deformation in the sense of Belgun \cite{Bel2} (see also a deformation of type II in \cite[Definition 7.5.9]{BG}).
An almost isomorphism is a diffeomorphism commuting with almost contact structures associated with Sasakian manifolds equivalently Sasakian isomorphism up to rescalings.

Combining this criterion and the Sasakian version of Simpson's correspondence between flat bundles and Higgs bundles (\cite{SimC,SimL}) proved in \cite{BK, BK2},
we obtain a criterion for uniformizing compact Sasakian manifolds in terms of stablities of Higgs bundles. 
For a Sasakian manifold $(M, T^{1,0}_{M},\eta)$, we have the canonical transverse holomorphic foliation ${\mathcal F}_{\xi}$ and hence we can define holomorphic structures and Higgs bundle structures on “basic” vector bundles over $(M,{\mathcal F}_{\xi})$.
We have a canonical Basic Higgs bundle structure $(E_{M},\, \theta_{M})$ on the vector bundle $E_{M}=T^{1,0}_{M}\oplus \C_{M}
$ where $T^{1,0}_{M}$ is the holomorphic tangent bundle associated with a CR structure and $\C_{M}$ is the trivial line bundle.

\begin{thm}\label{Iuni2}
We assume:
\begin{enumerate}
\item The basic Higgs bundle $(E_{M},\, \theta_{M})$ is stable.
\item The equality 
$$
\int_{M} \left(2c_{2, B}(T^{1,0}_{M}) -
\frac{n}{n+1}c_{1, B}(T^{1,0}_{M})^2\right)\wedge(d\eta)^{n-2}\wedge\eta \, =\, 0
$$
holds.
\item $c_{1, B}(T^{1,0}_{M})=-C[d\eta]$ for some positive constant $C$.
\end{enumerate}
Then the universal covering $\widetilde{M}$ with the lifting of some $A^1_{B}$-deformation of $( T^{1,0}_{M}, \eta)$ is almost isomorphic to the $\widetilde{S^{1}}(\bigwedge^m T^{1,0}_{D})$ with the lifting of the homogeneous Sasakian structure $(T^{1,0}_{D},\eta_{D})$ for $D=SU(n,1)/S(U(n)\times U(1))$.

\end{thm}
This result is based on the existence of Hermitian-Einstein metrics on stable basic Higgs bundles over compact Sasakian manifolds.
In \cite{Zh}, by using the existence of Sasaki-Einstein metrics, the result of uniformization by the odd-dimensional sphere is proved. 

As a corollary of Theorem \ref{Iuni2}, we obtain an analogue of the existence of a metric of constant negative curvature on a Riemannian surface of genus $\ge 2$.
\begin{cor}
Let $(M,T^{1,0}_{M},\eta)$ be a compact $3$-dimensional Sasakian manifold with $c_{1, B}(T^{1,0})=-C[d\eta]$.
Then the universal covering $\widetilde{M}$ with the lifting of some $A^1_{B}$-deformation of $( T^{1,0}_{M}, \eta)$ is almost isomorphic to $\widetilde{PU}(1,1)$ with a left-invariant Sasakian structure.
\end{cor}
We have $\widetilde{PU}(1,1)=\widetilde{SL}_{2}(\R)$.
This result is a part of Belgun's classification in \cite{Bel1, Bel2}.
Belgun uses the non-K\"ahler complex structure on $M\times S^{1}$ for a compact Sasakian manifold $M$.
On the other hand, we do not use such a complex structure.

\bigskip

\noindent {\bf Acknowledgement.} We are grateful to Professor Takuro Mochizuki for his valuable comments on the earlier works \cite{BK, BK2} which led to  the beginning of this project.

\section{Sasakian Geometry}
\subsection{Sasakian structures on CR manifolds}
Let $M$ be a $(2n+1)$-dimensional real smooth manifold. 
A {\em CR-structure} on $M$ is an 
$n$-dimensional complex involutive sub-bundle $T^{1,0}_{M}$ of the complexified tangent bundle $TM_{\C}\,=\, 
TM\otimes_{\mathbb R} {\C}$ such that $T^{1,0}_{M}\cap T^{0,1}_{M}=\{0\}$ 
where $T^{0,1}_{M}=\overline{T^{1,0}_{M}}$.
Define the real $2n$-dimensional sub-bundle $S=TM\cap (T^{1,0}_{M}\oplus T^{0,1}_{M})\subset TM$.
We have the almost complex structure $I: S\to S$ associated with the decomposition $ S_{\C}=T^{1,0}_{M}\oplus T^{0,1}_{M}$.

A {\em 
strongly pseudo-convex CR structure} on $M$ is a pair $(T^{1,0}_{M} , \eta) $
of a CR structure $T^{1,0}_{M}$ and a contact 
$1$-form $\eta$ such that $\ker\eta=S$ and the bilinear form defined by 
$L_{\eta}(X,Y)=d\eta(X, IY)$ on $S$ is a Hermitian metric on $(S, I)$.
Take the Reeb vector field $\xi$ associated with the contact $1$-form $\eta$.
$(T^{1,0}_{M} , \eta) $ is {\em Sasakian} if for any smooth section $X$ of $T^{1,0}_{M}$ the Lie bracket 
$[\xi, X]$ is also a smooth section of $T^{1,0}_{M}$.
We consider the $1$-dimensional foliation ${\mathcal F}_{\xi}$ on $M$ generated by $\xi$.
If $(T^{1,0}_{M} , \eta) $ is Sasakian, then $T^{1,0}_{M}$ defines a transverse holomorphic structure on ${\mathcal F}_{\xi}$ and $d\eta$ is a transverse K\"ahler form on ${\mathcal F}_{\xi}$.
We say that a Sasakian structure is quasi-regular if every leaf of ${\mathcal F}_{\xi}$ is closed.
If $M$ is compact and a Sasakian structure on $M$ is quasi-regular, then the flow of $\xi$ induces an action $S^{1}\times M\to M$ (see \cite{Wa}).

Let $(M, T^{1,0}_{M} , \eta) $ be a Sasakian manifold.
A differential form $\omega$ on $M$ is called {\em basic} if the equations
\[
i_{\xi}\omega=0= i_{\xi} d\omega
\]
hold.
We denote by $A^{\ast}_{B}(M)$ the subspace of basic 
forms in the de Rham complex $A^{\ast}(M)$.
Then
$A^{\ast}_{B}(M)$ is a sub-complex of the de Rham complex $A^{\ast}(M)$.
Denote by 
$H_{B}^{\ast}(M)$
the cohomology of the basic de Rham complex $A^{\ast}_{B}(M)$. 
We note that $d\eta\in A^{2}_{B}(M)$ and $[d\eta]\not=0\in H_{B}^{2}(M)$ if $M$ is compact.
We have the bigrading $A^{r}_{B}(M)_{\C}=\bigoplus_{p+q=r} A^{p,q}_{B}(M)$ as well as the decomposition of the exterior 
differential $d_{\vert A^{r}_{B}(M)_{\C}}=\partial_{B}+\overline\partial_{B}$ on $A^{r}_{B}(M)_{\C}$, so that $$\partial_{B}:A^{p,q}_{B}(M)\to
A^{p+1,q}_{B}(M)\ \text{ and } \
\overline\partial_{B}:A^{p,q}_{B}(M)\to 
A^{p,q+1}_{B}(M)\, .$$ 
We note that $d\eta\in A^{1,1}_{B}(M)$.
For a strongly pseudo-convex CR-manifold $(M, T^{1,0}_{M} , \eta)$
there exists a unique affine connection $\nabla^{TW}$ on $TM$ such that the following
conditions hold (\cite{Tan, Web}):
\begin{enumerate}
\item $S$ is parallel with respect to $\nabla^{TW}$.

\item $\nabla^{TW}I\,=\nabla^{TW}d\eta=\nabla^{TW}\eta\,=\nabla^{TW}\xi=\,0$.

\item The torsion $T^{TW}$ of the affine connection $\nabla^{TW}$ satisfies the equation
\[T^{TW} (X,\,Y)\,=\, -d\eta (X,\,Y)\xi
\]
for all $X,\,Y\,\in\, S_{x}$ and $x\,\in\, M$.
\end{enumerate}
This affine connection $\nabla^{TW}$ is called the {\em Tanaka--Webster connection}. It
is known that $ 
(T^{1,0},\eta)$ is a Sasakian manifold if and only if 
$T^{TW}(\xi,\,v)\,=\,0$ for all $v\, \in\, TM$.

Define the homomorphism $\Phi_{\xi}: TM\to TM $ by the extension of $I: S\to S$ satisfying $\Phi_{\xi}(\xi)=0$.
Then, $\Phi^{2}_{\xi}=-{\rm Id}+\xi\otimes \eta$.
We call $\Phi_{\xi}$ the {\em almost contact structure} associated with a Sasakian structure $(T^{1,0}_{M} , \eta) $.
We define the Riemannian metric $g_{\eta}(X,Y)=\eta(X)\eta(Y)+ d\eta(X, \Phi_{\xi} Y)$.
We call $g_{\eta}$ the {\em Sasakian metric} associated with a Sasakian structure $(T^{1,0}_{M} , \eta) $.

Let $(N,J)$ be a complex manifold.
A smooth map $f:M\to N$ is a holomorphic if $df \circ \Phi_{\xi}=J\circ df$.
A holomorphic map $f:M\to N$ satisfies $df (\xi)=0$.
We note that the holomorphic condition of a smooth map $f:M\to N$ depends only on the transverse holomorphic foliation ${\mathcal F}_{\xi}$.
A smooth function $f:M\to \C$ is holomorphic if and only if $f\in A_{B}^{0}(M)$ and $\bar\partial_{B}f=0$.

\subsection{Morphisms, equivalences and deformations}

Let $(M_{1}, T^{1,0}_{M_{1}},\eta_{1})$ and $(M_{2}, T^{1,0}_{M_{2}},\eta_{2})$ be two Sasakian manifolds.
A smooth map $f:M_{1}\to M_{2}$ is a {\em CR map} if $df ( T^{1,0}_{M_{1}})\subset T^{1,0}_{M_{2}}$.
a CR map $f:M_{1}\to M_{2}$ satisfies $f^{\ast}\eta_{2}=g(x)\eta_{1}$ for some function $g(x) $ on $M_{1}$.
By the standard arguments (see \cite[Lemma 2.3]{BaK}), we have:
\begin{lemma}
Let $f:M_{1}\to M_{2}$ be a diffeomorphism.
Then the following two conditions are equivalent:
\begin{itemize}
\item $f$ is a CR map and $f^{\ast}\eta_{2}=\eta_{1}$.
\item $f$ is isometry between $(M_{1}, g_{\eta_{1}})$ and $(M_{2},g_{\eta_{2}})$ and $df (\xi_{1})=\xi_{2}$.
\end{itemize}
\end{lemma}
A diffeomorphism satisfying the equivalent condition in this lemma is a {\em Sasakian isomorphism}.

A smooth map $f:M_{1}\to M_{2}$ is a {\em almost contact map} if $df \circ \Phi_{\xi_{1}}=\Phi_{\xi_{2}}\circ df$.
A almost contact map is a CR-map but the converse is not true.
We can easily check that an almost contact map $f:M_{1}\to M_{2}$ satisfies $f^{\ast}\eta_{2}=c\eta_{1}$ and $df (\xi_{1x})=c\xi_{2f(x)}$ for some constant $c $.
It is sufficient to check the above function $g(x)$ is constant.
By the basicness ${\mathcal L}_{\xi_{1}}d\eta_{1}={\mathcal L}_{\xi_{2}}d\eta_{2}=0$ and $df (\xi_{1x})=c(x)\xi_{2f(x)}$, we have
$${\mathcal L}_{\xi_{1}}(g)\eta_{1}={\mathcal L}_{\xi_{1}} (f^{\ast}\eta_{2}) =(i_{\xi_{1}}d+di_{\xi_{1}}) (f^{\ast}\eta_{2})=di_{\xi_{1}}f^{\ast}\eta_{2}=dg
$$
and so $dg=0$ on the sub-bundle $S={\rm ker} \eta_{1}\subset TM$.
It is sufficient to prove $dg (\xi_{1x})=0$ for any $x\in M$.
Since $d\eta$ is non-degenerate on $S$, we can take local sections $X, Y$ of $S$ such that $d\eta_{1} (X_{x},Y_{x})\not=0$.
By the property of the Tanaka-Webster connection $\nabla^{TW} $, 
we have 
\[\nabla^{TW}_{X}Y-\nabla^{TW}_{Y}X-[X,Y]= -d\eta (X,Y)\xi.
\] 
Since $\nabla^{TW}_{X}Y-\nabla^{TW}_{Y}X$ is a local section of $S$,
\[-d\eta_{1} (X,Y) dg (\xi_{1})=dg(\nabla^{TW}_{X}Y-\nabla^{TW}_{Y}X) -X(Y(g))+Y(X(g))=0.
\]
This implies $dg (\xi_{1x})=0$.

Let $(M, T^{1,0}_{M},\eta)$ be a Sasakian manifold.
A {\em rescaling} of $( T^{1,0}_{M},\eta)$ is a Sasakian structure $( T^{1,0}_{M\tau},R\eta)$ for a real number $R>0$.
The Reeb vector field of $R\eta$ is $\frac{1}{R}\xi$.
A {\em $A^1_{B}$-deformation} of $( T^{1,0}_{M},\eta)$ is a Sasakian structure $( T^{1,0}_{M\tau},\eta^{\tau})$ such that 
$\eta^{\tau}=\eta+\tau$ for $\tau\in A^1_{B}$ and $T^{1,0}_{M\tau}=\{X-\tau(X)\xi\vert X\in T^{1,0}_{M}\}$.
The Reeb vector field of $\eta^{\tau}$ is $\xi$ and the transverse holomorphic structure on ${\mathcal F}_{\xi}$ induced by $T^{1,0}_{M\tau}$ is same as the one induced by $T^{1,0}_{M}$.

By the above argument on almost contact maps and the standard arguments (see \cite[Proposition 8.1.1]{BG}), we have:
\begin{lemma}
Let $f:M_{1}\to M_{2}$ be a diffeomorphism.
Then the following two conditions are equivalent:
\begin{itemize}
\item $f$ is a Sasakian isomorphism up to rescalings.
\item $f$ is an almost contact map.
\end{itemize}

\end{lemma}
A diffeomorphism satisfying the equivalent condition in this lemma is an {\em almost isomorphism}.
Obviously, an almost isomorphism is a CR-diffeomorphism.

\section{Basic vector bundles over Sasakian manifolds}
\subsection{Basic vector bundles}
We define structures on smooth vector bundles over Sasakian manifolds $(M, T_{M}^{1,0}, \eta)$ related to ${\mathcal F}_{\xi}$ in terms of Rawnsley's partial flat connections in \cite{Ra}.

A structure of basic vector bundle on a $C^\infty$ vector bundle $E$ is a 
linear
operator $\nabla_{\xi}\,:\, {\mathcal C}^{\infty}(E)
\,\longrightarrow\,
{\mathcal C}^{\infty}(E)$ such that
\[ \nabla_{\xi}(f s)\,=\,f\nabla_{\xi} s+ \xi(f) s.
\]
Consider the flow $\R\times M \to M$ generated by $\xi$.
Then if we have a lifting action $\R\times E\ni (t,v)\mapsto \phi_{t}(v)\in E$ of the flow $\R\times M \to M$ on a smooth vector bundle $E$, then we 
define the basic bundle structure on $E$ by
\[ \nabla_{\xi}(s)_{x}=\frac{d}{dt}_{\vert t=0}\phi_{t}(v)_{x}.
\]

A differential form $\omega\,\in\, A^{\ast}(M,\,E)$ with 
values in $E$ is called basic if 
the equations
\[
i_{\xi}\omega=0=\nabla_{\xi} \omega
\]
hold where we extend $\nabla_{\xi}: A^{\ast}(M,\,E)\to A^{\ast}(M,\,E)$.
Let $A^{\ast}_{B}(M,\, E)\,\subset\,A^{\ast}(M,\,E)$ denote the 
subspace of basic forms in the space $A^{\ast}(M,\,E)$ of differential forms with values in 
$E$.
A Hermitian metric on $E$ is called {\em basic} if it is $\nabla_{\xi}$-invariant.
Unlike the usual Hermitian metric, a basic vector bundle may not admits a basic Hermitian metric in general.
For a connection operator $\nabla : A^{\ast}(M,\,E)\to A^{\ast+1}(M,\,E)$ such that the covariant derivative of $\nabla$ along $\xi$ is $\nabla_{\xi}$, the curvature $R^{\nabla}\in A^{2}(M,\, {\rm End} (E))$ is basic i.e. $R^{\nabla}\in A^{2}_{B}(M,\, {\rm End} (E))$.
For the Chern forms $c_{i}(E,\nabla)\in A^{2i}(M)$ associated with $\nabla$, we have $c_{i}(E,\nabla)\in A_{B}^{2i}(M)$ and we define the basic Chern classes $c_{i,B}(E)\in H^{2i}_{B}(M)$ by the cohomology classes of $c_{i}(E,\nabla)$.

A structure of basic holomorphic bundle on a $C^\infty$ vector bundle $E$ is a
linear
differential operator $\nabla^{\prime\prime}\,:\, {\mathcal C}^{\infty}(E)
\,\longrightarrow\,
{\mathcal C}^{\infty}(E\otimes ( \langle \xi\rangle \oplus T^{0,1})^{\ast})$ such that
\begin{itemize}
\item for any $X\in \langle \xi\rangle \oplus T^{0,1}$, and any smooth function $f$ on $M$, the equation
\[ \nabla^{\prime\prime}_{X}(f s)\,=\,f\nabla^{\prime\prime}_{X} s+ X(f) s
\]
holds for all smooth sections $s$ of $E$, and
\item if we extend $\nabla^{\prime\prime}$ to $\nabla^{\prime\prime}\,:\, {\mathcal C}^{\infty}(E\otimes \bigwedge^{k} (\langle \xi\rangle \oplus T^{0,1})^{\ast}) \, \longrightarrow\, {\mathcal C}^{\infty}(E\otimes 
\bigwedge^{k+1} (\langle \xi\rangle \oplus T^{0,1})^{\ast})$, then $\nabla^{\prime\prime}\circ \nabla^{\prime\prime}\,=\,0$.
\end{itemize}

A basic holomorphic vector bundle $E$ has a canonical basic bundle structure corresponding to the derivative $\nabla^{\prime\prime}_{\xi}$.
$\nabla^{\prime\prime}$ defines the linear operator $\bar\partial_{E}: A^{p,q}_{B}(M,\, E)\to 
A^{p,q+1}_{B}(M,\, E)$ so that 
$\bar \partial_{E}( f\omega )=\bar\partial _{B}f \wedge \omega +f \bar\partial_{E}( \omega )$ for $f\in A^{0}_{B}(M), \omega\in A^{p,q}_{B}(M,\, E)$.
If a basic holomorphic vector bundle $E$ admits a basic Hermitian metric $h$, as complex case, we have a unique unitary connection $\nabla^{h}$ such that for any $X\in \langle \xi\rangle \oplus T^{0,1}$, $\nabla^{h}_{X}=\nabla^{\prime\prime}_{X}$.

\begin{example}\label{line}
Consider the ${\mathcal C}^{\infty}$-trivial complex line bundle $E\,=\,M\times \C\, 
\longrightarrow\, M$. For any $C\in\R$, we define the connection 
$
\nabla^{C}\,=\,d-2\pi\sqrt{-1}C\eta
$
on $E$.
Then, the curvature of $\nabla^{C}$ is $Cd\eta$. Since $d\eta\,\in\, 
A^{1,1}_{B}(M)$, this $\nabla^{C}$ induces a structure of a basic holomorphic bundle.
Consequently, we have a non-trivial holomorphic vector bundle structure $E_{C}$ on $E$ 
that depends on $C\,\in\, \C$. The basic cohomology class of $Cd\eta$ 
is the basic first Chern class of the basic vector bundle $E_{C}$. Thus 
$\{E_{C}\}_{C\in\R}$ is a family of basic vector bundles such that $E_{C}\,\not\cong\, 
E_{C^{\prime}}$ for every $C\,\not=\,C^{\prime}$. 
The standard constant Hermitian metric $h$ on the 
${\mathcal C}^{\infty}$ trivial line bundle $E\,=\,M\times \C$ is basic on $E_{C}$. 
The 
connection $\nabla^{C}$ is unitary for $h$, and hence $\nabla^{C}$ is the 
canonical connection for the basic Hermitian metric $h$.

Consider the flow $\R\times M \to M$ generated by $\xi$.
Define the lifting action $\R\times M\times \C\ni (t, x, v)\mapsto (x_{t}, e^{-2\pi C\sqrt{-1}t}v)\in E$ of the flow $\R\times M \to M$.
Then this induces the basic bundle structure $\nabla^{C}_{\xi}$ of $E_{C}$. 
\end{example}

\begin{example}
The Tanaka-Webster connection $\nabla^{TW} $ defines a 
structure of a basic holomorphic bundle on $T^{1,0}$.
Since we have $\nabla^{TW} _{\xi}(X)=[\xi,X]$,
the basic bundle structure $\nabla^{TW}_{\xi}$ is given by the action $\R\times T^{1,0}X\to T^{1,0}X$ which is the differential of the action $\R\times M\to M$ generated by the flow of $\xi$.

\end{example}

\subsection{Equivariant cohomology}
Assume that 
$M$ is compact and a Sasakian structure on $M$ is quasi-regular.
Consider the equivariant cohomology $H^{\ast}_{S^{1}}(M,R)=H^{\ast}(M\times_{S^{1}} ES^{1},R)$ for the
action $S^{1}\times M\to M$ with coefficients in a commutative ring $R$.
$H^{\ast}_{S^{1}}(M,\R)$ is the cohomology of the Cartan model $A^{\ast}(M)^{S^{1}}\otimes \R[u]$ with the differential $d_{S^{1}}$
so that $d_{S^{1}} u=0$, $d_{S^{1}}\omega= d\omega -i_{\xi}\omega u$ where $u$ is of degree $2$.
The natural inclusion $A^{\ast}_{B}(M)\to A^{\ast}(M)^{S^{1}}\otimes \R[u]$ is a cochain complex homomorphism.
Since the action $S^{1}\times M\to M$ is locally free, the induced map $H_{B}^{\ast}(M)\to H^{\ast}_{S^{1}}(M,\R)$ is an isomorphism (\cite[Section 5]{GS}).

We use the result in \cite{Mu} which says that equivariant first Chern classes in $H^{\ast}_{S^{1}}(M,\Z)$ classify lifts of the action $S^{1}\times M\to M$ on line bundles.
For our case, basic bundle structures $\nabla_{\xi}$ on line bundles can be considered as infinitesimal lifts of the action $S^{1}\times M\to M$.
Let $L_{1}\to M, L_{2}\to M$ be line bundles admitting lifts of the action $S^{1}\times M\to M$ and consider them as basic line bundles.
Assume that we have basic Hermitian metrics $h_{1}, h_{2}$ and basic unitary connections $\nabla_{1}, \nabla_{2}$ on $L_{1}$ and $L_{2}$ respectively.
We note that the equivariant first Chern classes in $H^{\ast}_{S^{1}}(M,\R)$ are images of the basic first Chern classes via the canonical map $H_{B}^{\ast}(M)\to H^{\ast}_{S^{1}}(M,\R)$.
Since $H_{B}^{\ast}(M)\to H^{\ast}_{S^{1}}(M,\R)$ is an isomorphism, we can say that 
if $c_{1,B}(L_{1})=c_{1,B}(L_{2})$, then $L_{1}^{l}$ and $L_{2}^{l}$ are equivariantly isomorphic for some positive integer $l$.

\begin{prop}\label{equiiso}
If $c_{1, B}(T^{1,0}_{M})=-C[d\eta]$ for some positive constant $C$, then for some positive integer $l$, the $l$-th power of  the anti-canonical bundle $\bigwedge^{n} T^{1,0}_{M}$ with the $S^{1}$-action given by the differential of the 
action $S^{1}\times M\to M$ is equivariantly isomorphic to $E_{-lC}$.
\end{prop}

\begin{remark}\label{Neg}
If $c_{1,B}(T^{1,0}_{M})=-C[d\eta]$ for a positive constant $C$, then the Sasakian structure $(T^{1,0}_{M}, \eta)$ is quasi-regular (\cite[Theorem 8.1.14]{BG}).
\end{remark}

\section{Variations of Hodge structure and uniformizations}
\subsection{Variations of Hodge structure over Sasakian manifolds}
Let $G$ be a connected semi-simple Lie group.
Assume that $G$ has the finite center and no compact factor.
A Hodge structure on $G$ is a real Hodge structure $\g_{\C}=\bigoplus_{p\in \Z}\g^{p,-p}$ of weight $0$ on the Lie algebra $\g$ of $G$ such that the Lie bracket on $\g$ is a morphism of real Hodge structures and $-B_{\g}$ is a polarization where $B_{\g}$ is the Killing form on $\g$.
Define the subalgebra $\frak{v}=\g^{0,0}\cup \g \subset \g$
and the subgroup $V=\exp(\frak{v})\subset G$.
Then $V$ is compact.
The subalgebra $\bigoplus_{p>0} \g^{ p.-p} \subset \g_{\C}$ defines a complex structure on $D=G/V$.
The complex manifold $D$ is called the {\em period domain}.
Define the subalgebra $\frak{p}=\bigoplus_{p\ge 0} \g^{ p.-p} $ and the subgroup $P=\exp(\frak{p})\subset G_{\C}$ in the complexification $G_{\C}$ of $G$.
Then $\frak{p}$ is a parabolic subalgebra in $\g_{\C}$ and $G/V$ is a dual manifold of $G/P$ in the sense of \cite{GrSc}.
The holomorphic tangent bundle $T^{1,0}_{D}$ is a homogeneous vector bundle given by $(G\times \bigoplus_{p>0} \g^{ p.-p})/V$ and hence we have the decomposition $T^{1,0}_{D}=\bigoplus_{p>0} E^{p,-p}$.
Consider the connection form $\lambda$ on the anti-canonical bundle $\bigwedge^m T^{1,0}_{D}$ given by the Maurer-Cartan form where $m=\dim_{\C} D$.
Let $\eta_{D}=-\lambda$. Then $d\eta$ is a non-degenerate $2$-form on $D$ which is positive on $\bigoplus_{p\, odd} E^{p,-p}$ and negative on $\bigoplus_{p\, even } E^{p,-p}$ see \cite[(4.23)]{GrSc}.
Let $(M,T^{1,0}_{M}, \eta)$ be a Sasakian manifold.
Denote by $\widetilde{M}$ the universal cover of $M$ at $x\in M$ and take the lifting Sasakian structure $(T^{1,0}_{\tilde{M}}, \tilde\eta)$.
A {\em $G$-variations of Hodge structure} is $(G, \rho, h)$ such that: 
\begin{enumerate}
\item $G$ is a connected semi-simple Lie group equipped with a Hodge structure.
\item $\rho:\pi_{1}(M,x)\to G$ is a representation.
\item $h: \widetilde{M}\to D$ is a $\rho$-equivariant holomorphic map satisfying $dh(T^{1,0}_{\widetilde{M}})\subset E^{1,-1}$.
\end{enumerate} 

\subsection{Uniformizing Variations of Hodge structure}
A Hodge structure on $G$ is {\em Hermitian} if $\g_{\C}=\g^{1,-1}\oplus \g^{0,0}\oplus \g^{-1,1}$.
In this case, $D$ is a Hermitian symmetric manifold of non-compact type and $d\eta_{D}$ is an invariant K\"ahler form on $D$.
Regarding $\eta$ as a contact $1$-form, $(T^{1,0}_{D},\eta_{D})$ is a $G$-invariant Sasakian structure on the circle bundle $S^{1}(\bigwedge^m T^{1,0}_{D})$.
We have $S^{1}(\bigwedge^m T^{1,0}_{D})=(G\times S^{1})/V=G/V^{\prime}$ where $V^{\prime}$ is the kernel of the non-trivial representation $V\to GL(\bigwedge^{m}\g^{1,-1})$.
$S^{1}(\bigwedge^m T^{1,0}_{D})$ is also $G$-homogeneous and the Sasakian structure $(T^{1,0}_{D},\eta_{D})$ is homogeneous
We can take the Reeb vector field $\xi_{D}$ as a fundamental vector field of the action $S^{1}\times (G\times S^{1})/V \ni (a, [g,t])\mapsto [g, a^{-1}t]\in (G\times S^{1})/V$.
Since $D$ is contractible, by trivializing $S^{1}(\bigwedge^m T^{1,0}_{D})\cong S^{1}\times D$, we have the universal covering $\widetilde{S^{1}}(\bigwedge^m T^{1,0}_{D})\to S^{1}(\bigwedge^m T^{1,0}_{D})$ with the infinite cyclic covering group.

Let $\Gamma\subset G$ be a discrete subgroup which acts freely and cocompactly on $S^{1}(\bigwedge^m T^{1,0}_{D})=G/V^{\prime}$.
We consider the compact Sasakian manifold $(\Gamma \backslash S^{1}(\bigwedge^m T^{1,0}_{D}), T^{1,0}_{D} \eta_{D})$.
We observe that:
\begin{itemize}
\item By the construction, we have $c_{1,B}(M)=-[d\eta_{D}]$.
\item Defining $h: \widetilde{S^{1}}(\bigwedge^m T^{1,0}_{D})\to D$ by the composition of the covering $\widetilde{S^{1}}(\bigwedge^m T^{1,0}_{D})\to S^{1}(\bigwedge^m T^{1,0}_{D})$ and the projection $ S^{1}(\bigwedge^m T^{1,0}_{D})\to D$ and $\rho: \pi_{1}(\Gamma \backslash S^{1}(\bigwedge^m T^{1,0}_{D}))\to G$ by the composition the quotient $\pi_{1}(\Gamma \backslash S^{1}(\bigwedge^m T^{1,0}_{D}))\to \Gamma$ by the covering group and the inclusion $\Gamma \to G$, we have the $G$-variation of Hodge structure $(G, \rho,h)$. 

\end{itemize} 
The main purpose of this section is to prove that these properties uniformize compact Sasakian manifolds by the homogeneous Sasakian structure

Let $(M,T^{1,0}_{M}, \eta)$ be a Sasakian manifold.
A $G$-variation of Hodge structure $(G, \rho, h)$ is {\em uniformizing} if 
a Hodge structure on $G$ is Hermitian and the differential $dh: T^{1,0}_{\widetilde{M}} \to T^{1,0}_{D}$ is an isomorphism at every fiber.

\begin{thm}
Let $(M, T^{1,0}_{M}, \eta)$ be a compact Sasakian manifold.
If $(M,T^{1,0}_{M}, \eta)$ admits a uniformizing $G$-variations of Hodge structure $(G,\rho, h)$ and $c_{1,B}(T_{M}^{1,0}
)=-C[d\eta]$ for a positive constant $C$, then
the universal covering $\widetilde{M}$ with the lifting of some $A^1_{B}$-deformation of $( T^{1,0}_{M}, \eta)$ is almost isomorphic to $\widetilde{S^{1}}(\bigwedge^m T^{1,0}_{D})$ with the lifting of the homogeneous Sasakian structure $(T^{1,0}_{D},\eta_{D})$.

\end{thm}
\begin{proof}
By Proposition \ref{equiiso} and Remark \ref{Neg},  for some positive integer $l$, the $l$-th power of  $\bigwedge^n T^{1,0}_{M}$ is equivariantly isomorphic to $E_{-lC}$.
Consider the basic holomorphic line bundle $L$ over $M$ defined by the pull-back $h^{\ast}\bigwedge^m T^{1,0}_{D}$.
Then, Since $f$ is $\rho$-equivariant holomorphic and $dh: T^{1,0}_{\tilde{M}} \to T^{1,0}_{D}$ is an isomorphism at every fiber, $L$ is identified with $\bigwedge^m T^{1,0}_{M}$ via the differential $dh$.
We regard the $\rho$-equivariant map $h: \widetilde{M}\to D$ as a $V$-reduction $P$ of the flat bundle $(\widetilde{M}\times G)/\pi_{1}(M,x)$.
Then $L$ is induced by $P$ associated with the representation $V\to GL(\bigwedge \g^{1,-1})$.
Take an equivariant isomorphism $L^l\cong E_{-lC}$.
By $E_{-lC}=M\times \C$, corresponding to $M\ni x\mapsto (x,1)\in M\times \C$,
we have a nowhere vanishing section of $L^l$ and it induces an $\rho$-equivariant map $f: \widetilde{M}\to (G\times S^1)/V=(G\times S^{1})/V=G/V^{\prime\prime}$ where $V^{\prime\prime}$ is the kernel of the  representation $\Lambda^{l}$. 
$f$ is a lift of $h: \widetilde{M}\to D$.
$G/V^{\prime}$ is a $l$-sheeted covering of $G/V^{\prime\prime}$.
$(T^{1,0}_{D},\eta_{D})$ defines a Sasakian structure on $G/V^{\prime\prime}$.
By the equivariance of $L^l\cong E_{-lC}$, we have $df (\tilde\xi)=lC\xi_{D}$.
Since $M$ is compact and $f$ is $\rho$-equivariant, the Riemann metric $f^{\ast}g_{\eta_{D}}$ is complete on $\widetilde{M}$.
$f:\widetilde{M}\to G/V^{\prime\prime}$ is a local isometry and hence covering map.

Let $\tau=-\tilde\eta+\frac{1}{C}f^{\ast}\eta_{D}$.
Then, by $\tau(\tilde\xi)=0$ and $d\tau\in A^{2}_{B}(\widetilde{M})$, we have $\tau\in A^{1}_{B}(\widetilde{M})$.
The $A^{1}_{B}$-deformation $(T^{1,0}_{\widetilde{M}\tau}, \tilde\eta+\tau)$ of $(T^{1,0}_{\widetilde{M}}, \tilde\eta)$ is almost isomorphic to $(\widetilde {S^{1}}(\bigwedge^m T^{1,0}_{D}), T^{1,0}_{D},\eta_{D})$.
Since $f$ is $\rho$-equivariant, $(T^{1,0}_{\widetilde{M}\tau}, \tilde\eta+\tau)$ is defined on $M$.
\end{proof} 

\begin{remark}
In this proof, $f:\widetilde{M}\to G/V^{\prime}$  satisfies $df(T^{1,0}_{\widetilde{M}})=T^{1,0}_{D}$ modulo $\langle \xi\rangle $ but may not be a CR-map. Hence, the $A_{B}^1$-deformation $(T^{1,0}_{\widetilde{M}\tau}, \tilde\eta+\tau)$ is essential.
\end{remark}

\section{Higgs bundles and uniformizations}
\subsection{Higgs bundles over Sasakian manifolds}
Let $(M,T^{1,0}_{M}, \eta)$ be a compact Sasakian manifold.
A \textit{basic Higgs bundle} over $M$ is a pair
$(E, \,\theta)$ consisting of a basic holomorphic vector bundle $E$
and 
$\theta\,\in\, A^{1,0}_{B}(M,\,{\rm 
End}(E))$ satisfying the following two conditions:
$$\overline\partial_{{\rm 
End}(E)}\theta \,=\,0\ \ \text{ and }
\ \ \theta\wedge \theta\,=\,0\, .$$

We define the {\it degree} of a basic holomorphic vector bundle $E$ by 
$${\rm deg}(E)\,:=\,\int_{M}c_{1, B}(E)\wedge (d\eta)^{n-1}\wedge\eta\, .$$
Denote by ${\mathcal O}_{B}$ the sheaf of holomorphic functions on $M$, and for a holomorphic vector bundle $E$ on $M$, denote by 
${\mathcal O}_{B}(E)$ the sheaf of holomorphic sections of $E$. 
Consider ${\mathcal 
O}_{B}(E)$ as a coherent ${\mathcal O}_{B}$-sheaf.

For a basic Higgs 
bundle $(E,\, \theta)$, a {\em sub-Higgs sheaf} of $(E,\, \theta)$ is a coherent 
${\mathcal O}_{B}$-subsheaf $\mathcal V$ of ${\mathcal O}_{B}(E)$ such that $\theta ({\mathcal V})\, \subset\, {\mathcal 
V}\otimes \Omega_{B}$, where $\Omega_{B}$ is the sheaf 
of basic holomorphic $1$-forms on $M$. By \cite[Proposition 3.21]{BH}, if ${\rm rk} 
(\mathcal V)\,<\,{\rm rk}(E)$ and ${\mathcal O}_{B}(E)/\mathcal V$ is 
torsion-free, then there is a transversely analytic sub-variety $S\,\subset\, M$ of complex 
co-dimension at least 2 such that ${\mathcal V}\big\vert_{M\setminus S}$ is given by a basic holomorphic sub-bundle $V\,\subset\,
E\big\vert_{M\setminus S}$. The degree ${\rm deg}(\mathcal V)$ can be defined by integrating $c_{1, B}(V)\wedge (d\eta)^{n-1}\wedge\eta$
on this complement $M\setminus S$.

\begin{dfn}
We say that a basic Higgs bundle $(E,\, \theta)$ is {\em stable} if $E$ admits a basic Hermitian metric and for every 
sub-Higgs sheaf ${\mathcal V}$ of $(E,\, \theta)$ such that ${\rm rk} (\mathcal V)\,<\,{\rm 
rk}(E)$ and ${\mathcal O}_{B}(E)/\mathcal V$ is torsion-free, 
the inequality
\[\frac{{\rm deg}(\mathcal V)}{{\rm rk} (\mathcal V)}\,<\,\frac{{\rm deg}(E)}{{\rm rk}(E)}
\]
holds.

A basic Higgs bundle $(E,\, \theta)$ is called {\em polystable} if
$$
(E,\, \theta)\,=\, \bigoplus_{i=1}^k (E_i,\, \theta_i)\, ,
$$
where each $(E_i,\, \theta_i)$ is a stable Higgs bundle with
\[\frac{{\rm deg}( E_i)}{{\rm rk} ( E_i)}\,=\,\frac{{\rm deg}(E)}{{\rm rk}(E)}\, .
\]
\end{dfn}

Let $(E, \,\theta)$ be a basic Higgs bundle over a compact Sasakian manifold $M$. 
Assume that $E$ admits a basic Hermitian metric $h$.
Define $\overline\theta_{h}\,\in\, A^{0,1}_{B}(M,\,{\rm End}(E))$ by
$
h(\theta (e_{1}),\, e_{2})\,=\,h(e_{1}, \,\overline\theta_{h} (e_{2}))
$
for all $e_{1},\,e_{2}\,\in\, E_x$ and all $x\, \in\, M$.
Define the canonical connection
$
D^{h}\,=\,\nabla^{h}+\theta+\overline\theta_{h}
$
on $E$.

\begin{thm}[\cite{BK, BK2}]\label{harmonicmetric}
If a basic Higgs bundle $(E, \,\theta)$ is stable and satisfies
$$c_{1,B}(E)\,=\,0\qquad {\rm and} \qquad \int_{M} c_{2,B}(E)\wedge(d\eta)^{n-2}\wedge \eta\,=\,0,$$
then there exists a basic Hermitian metric $h$ so that the canonical connection $D^{h}$ is flat 
and such Hermitian metric is unique up to a positive constant.
Moreover, such metric $h$ is a Harmonic metric on the flat bundle $(E,D^{h})$ with respect to the Sasakian metric $g_{\eta}$ and hence the flat bundle $(E,D^{h})$ is semi-simple (\cite{Cor}).

\end{thm}
This hermitian metric is said to be a {\em harmonic metric} on $(E,\theta)$.

A {\em polarized complex basic variation of Hodge structure of weight $w$} over $M$ is $(E=\bigoplus_{p+q=w}E^{p,q},D,H)$ so that 
\begin{enumerate}
\item 
$(E,D)$ is a complex flat vector bundle.
\item $\bigoplus_{p+q=w}E^{p,q}$ is a direct sum of $C^{\infty}$-subbundles such that $D_{\xi} {\mathcal C}^{\infty}(E^{p,q})\subset {\mathcal C}^{\infty}(E^{p,q})$.
\item (The Griffiths transversality conditions) For any $X\in T^{1,0}_{M}$ (resp. $X\in T^{0,1}_{M}$), $D_{X}{\mathcal C}^{\infty}(E^{p,q})\subset {\mathcal C}^{\infty}(E^{p,q})\oplus {\mathcal C}^{\infty}(E^{p-1,q+1}) $
(resp. $D_{X}{\mathcal C}^{\infty}(E^{p,q})\subset {\mathcal C}^{\infty}(E^{p,q})\oplus {\mathcal C}^{\infty}(E^{p+1,q-1}) $)

\item $H$ is 
a parallel Hermitian form so that
decomposition $\bigoplus_{p+q=w}E^{p,q}$ is orthogonal and $h$ is positive on $E^{p,q}$ for even $p$ and negative for odd $p$.
\end{enumerate}
Define the basic bundle structure $\nabla_{\xi}$ on $E$ and each $E^{p,q}$.
By the flatness $D^2=0$ and the Griffiths transversality conditions, we have:
\begin{itemize}
\item
We have $D=\nabla+\theta+\bar\theta$ such that $\nabla $ is a connection on each $E^{p,q}$ and $\theta\in A^{1}(M, {\rm Hom}(E^{p,q}, E^{p-1,q+1}))$, $\bar\theta\in A^{1}(M, {\rm Hom}(E^{p,q}, E^{p+1,q-1}))$.
\item 
$\nabla $ gives a basic holomorphic bundle structure $\nabla^{\prime\prime}$, $\theta\in A^{1,0}(M, {\rm Hom}(E^{p,q}, E^{p-1,q+1}))$, $\bar\partial_{{\rm 
End}(E)}\theta=0$ and 
$\theta\wedge \theta=0$.
Thus $(E, \theta)$ is a basic Higgs bundle.
\item Define the basic Hermitian metric $h$ on $E$ by $(-1)^pH$ on each $E^{p,q}$. Then $\nabla$ is the canonical connection associated with $h$ and $\bar\theta=\overline\theta_{h}$ as the above sense.
Thus $h$ is a Harmonic metric on a basic Higgs bundle $(E,\theta)$.

\end{itemize}
Let $G=SU(E_{x}, H_{x})$ and $V=S(\Pi_{p+q=w} U(E^{p,q}_{x}, H_{x}))$.
Take the monodromy $\rho :\pi_{1}(M,x)\to GL(E_{x})$ of the flat bundle $(E,D)$.
Since $h$ is parallel, we can assume $\rho(\pi_{1}(M,x))\subset G$.
Let $\g$ be the Lie algebra of $G$.
The decomposition $E=\bigoplus_{p+q=w}E^{p,q}$ gives the decomposition $\g_{\C}=\bigoplus_{r}\g^{r,-r}$.
This decomposition is a Hodge structure on $G$ and the period domain is $G/V$ (cf. \cite{Sch}).
Considering $h$ as a $V$-reduction of the flat bundle $(\tilde{M}\times G)/\pi_{1}(M,x)$, we have a $\rho$-equivariant map $h: \tilde{M}\to G/V$ such that $dh= \theta$ on $T^{1,0}_{\tilde M}$.
By the Griffiths transversality conditions, $(G, \rho, h)$ is a $G$-variation of Hodge structure.

Let $(E, \,\theta)$ be a polystable basic Higgs bundle satisfying
$$c_{1,B}(E)\,=\,0\qquad {\rm and} \qquad \int_{M} c_{2,B}(E)\wedge(d\eta)^{n-2}\wedge \eta\,=\,0.$$
By Theorem \ref{harmonicmetric}, we have a harmonic metric $h$ on $(E,\theta)$.
For any $t\in U(1)$, we have the new polystable basic Higgs bundle $(E, \,t\theta)$ and we can easily check that $h$ is also a
harmonic metric on $(E,\, t \theta)$.
Assume that $U(1)$ acts on $E$ by holomorphic automorphisms via $\mu : U(1)\to {\rm Aut} (E)$ so that $\mu(t)\theta \mu(t)^{-1}=t\theta $.
Then, $\mu(t)$ is an isomorphism between basic Higgs bundles $(E, \,\theta)$ and $(E, \,t\theta)$.
By the harmonicity of $h$ on both $(E, \,\theta)$ and $(E, \,t\theta)$ and the uniqueness of harmonic metrics,
we can say that the action $\mu : U(1)\to {\rm Aut} (E)$ preserves a harmonic metric $h$.
Fix $w\in \Z$, define the sub-bundle $E^{p,w-p}$ consisting of $e\in E$ such that for any $t\in U(1)$, $\mu(t)e=t^{p}e$.
The polarized complex basic variation of Hodge structure of weight $w$ $(E=\bigoplus_{p+q=w}E^{p,q},D^{h},H)$ 
is defined by $H=(-1)^{p}h $ on each $E^{p,q}$.

\subsection{Uniformizations associated with the canonical Higgs bundles}
Define the Higgs bundle $(E_{M},\, \theta_{M})$ by the following way:
\begin{itemize}
\item $E_{M}\,=\,\C_{M}\oplus T^{1,0}_{M}$ where $\C_{M}$ is the trivial basic holomorphic line bundle on $M$, and
\item $\theta_{M}\,=\,\left(
\begin{array}{cc}
0&0 \\
1 & 0
\end{array}
\right)$ where $1$ is the identity element in ${\rm End} (T^{1,0}_{M})=(T^{1,0}_{M})^{\ast}\otimes T^{1,0}_{M}$ regarded as an element in $A^{1,0}_{B}(M,{\rm Hom}(\C_{M}, T^{1,0}_{M}))$.
\end{itemize}

\begin{thm}
We assume:
\begin{enumerate}
\item The basic Higgs bundle $(E_{M},\, \theta_{M})$ is stable.
\item The equality 
$$
\int_{M} \left(2c_{2, B}(T^{1,0}_{M}) -
\frac{n}{n+1}c_{1, B}(T^{1,0}_{M})^2\right)\wedge(d\eta)^{n-2}\wedge\eta \, =\, 0
$$
holds.
\item $c_{1, B}(T^{1,0}_{M})=-C[d\eta]$ for some positive constant $C$.
\end{enumerate}
Then the universal covering $\widetilde{M}$ with the lifting of some $A^1_{B}$-deformation of $( T^{1,0}_{M}, \eta)$ is almost isomorphic to $\widetilde{S^{1}}(\bigwedge^m T^{1,0}_{D})$ with the lifting of the homogeneous Sasakian structure $(T^{1,0}_{D},\eta_{D})$ for $D=SU(n,1)/S(U(n)\times U(1))$.

\end{thm}
\begin{proof}
By $c_{1, B}(T^{1,0}_{M})=-C[d\eta]$, $c_{1,B}(E\otimes E_{C^{\prime}})=0$ for some $C^{\prime}\in \R$.
By 
$$
\int_{M} \left(2c_{2, B}(T^{1,0}_{M}) -
\frac{n}{n+1}c_{1, B}(T^{1,0}_{M})^2\right)\wedge(d\eta)^{n-2}\wedge\eta \, =\, 0,
$$
we have
$ \int_{M} c_{2,B}(E_{M}\otimes E_{C^{\prime}})\wedge(d\eta)^{n-2}\wedge \eta\,=\,0.$
$\nabla^{\prime\prime}=\nabla_{\vert \langle \xi\rangle \oplus T^{0,1}}$ defines a holomorphic bundle structure on $L$.
Thus, $(E_{M}\otimes E_{C^{\prime}},\, \theta\otimes {\rm Id }_{L})$ is a basic Higgs bundle.
Since $(E_{M},\, \theta_{M})$ is stable, $(E_{M}\otimes E_{C^{\prime}}\, \theta_{M}\otimes {\rm Id }_{E_{C^{\prime}}})$ is also stable.
Let $h$ be a harmonic metric on $(E\otimes E_{C^{\prime}},\, \theta\otimes {\rm Id }_{E_{C^{\prime}}})$.
Define $E^{0,1}=T^{1,0}\otimes E_{C^{\prime}}$ and $E^{1,0}=E_{C^{\prime}}$.
We have the polarized complex basic variation of Hodge structure of weight $1$ $(E_{M}\otimes E_{C^{\prime}}=E^{1,0}\oplus E^{0,1},D^{h},H)$.
$SU(E_{x}, H_{x})\cong SU(n,1)$ and the Hodge structure on $SU(n,1)$ is Hermitian and the period domain $D$ is the unit ball $SU(n,1)/ S(U(n)\times U(1))$.
We obtain a $SU(n,1)$-variation of Hodge structure $(SU(n,1), \rho, h)$.
By the definition, for the map $h: \tilde{M}\to SU(n,1)/ S(U(n)\times U(1))$, $dh=\theta$ is an isomorphism on each fiber of $T^{1,0}_{\widetilde{M}}$.
Hence, $(SU(n,1), \rho, h)$ is uniformizing.
\end{proof}

\begin{remark}
If the basic Higgs bundle $(E_{M}, \theta_{M})$ is stable,
then the Miyaoka-Yau type inequality
$$
\int_{M} \left(2c_{2, B}(T^{1,0}_{M}) -
\frac{n}{n+1}c_{1, B}(T^{1,0}_{M})^2\right)\wedge(d\eta)^{n-2}\wedge\eta \, \geq\, 0\, .
$$
holds.
Thus the second condition means that the equality of the Miyaoka-Yau type inequality holds.

In \cite{Zh}, Zhang proves that if $c_{1, B}(T^{1,0}_{M})=C[d\eta]$ for some positive constant $C$ and the certain condition in terms of Sasaki-Einstein geometry holds, then the Miyaoka-Yau type inequality holds and the equality is a criterion for uniformizing a higher dimensional compact Sasakian manifold by the odd-dimensional sphere.
\end{remark}

\begin{remark}
In case $D=SU(n,1)/S(U(n)\times U(1))$,
the line bundle $\bigwedge^{n} T^{1,0}_{D}$ is a homogeneous vector bundle corresponding to the representation $S(U(n)\times U(1))\ni \left(
\begin{array}{cc}
U & 0 \\
0 & {\rm det}(U)^{-1}
\end{array}
\right)\mapsto {\rm det}(U)^{n+1}$.
Define the subgroup $U^{\prime}(n)=\left\{\left(
\begin{array}{cc}
U & 0 \\
0 & {\rm det}(U)^{-1}
\end{array}
\right): {\rm det}(U)^{n+1}=1\right\}$ in $SU(n,1)$.
We have $ S^{1}( \bigwedge^{n} T^{1,0}_{D})=SU(n,1)/U^{\prime}(n)=PU(n,1)/ SU(n)$.
Taking the universal covering $\widetilde{PU}(n,1)$ of $PU(n,1)$, we have $\widetilde{S^{1}}(\bigwedge^m T^{1,0}_{D})=\widetilde{PU}(n,1)/SU(n)$.
The lifting of the Sasakian structure $(T_{D}^{1,0},\eta_{D} )$ is $\widetilde{PU}(n,1)$-invariant.
\end{remark}

Assume $n=1$.
Then $H^{2}_{B}(M)=\langle [d\eta]\rangle$.
In this case, if $c_{1, B}(T^{1,0})=-C[d\eta]$ for some positive constant $C$, then the Higgs bundle $(E_{M},\, \theta_{M})$ is stable.
By $H^{4}_{B}(M)=0$, the equality 
$$
\int_{M} \left(2c_{2, B}(T^{1,0}) -
\frac{n}{n+1}c_{1, B}(T^{1,0})^2\right)\wedge(d\eta)^{n-2}\wedge\eta \, =\, 0
$$
is trivial.
\begin{cor}
Let $(M,T^{1,0}_{M},\eta)$ be a compact $3$-dimensional Sasakian manifold with $c_{1, B}(T^{1,0})=-C[d\eta]$.
Then the universal covering $\widetilde{M}$ with the lifting of some $A^1_{B}$-deformation of $( T^{1,0}_{M}, \eta)$ is almost isomorphic to $\widetilde{PU}(1,1)$ with a left-invariant Sasakian structure.
\end{cor}

\end{document}